\documentclass[12pt]{amsart}

\usepackage{amscd, amsfonts, amsmath, amssymb, amsthm, mathrsfs, appendix, enumerate, graphicx, flexisym, epsfig, float, graphicx, hyperref, pdfpages, tikz, mathtools, comment, tikz-cd}
\usepackage{multicol}
\usepackage[all,cmtip]{xy}
\usepackage{graphicx}

\makeatletter
\@namedef{subjclassname@2020}{Mathematics Subject Classification \textup{2020}}

\newsavebox{\@brx}
\newcommand{\llangle}[1][]{\savebox{\@brx}{\(\m@th{#1\langle}\)}%
  \mathopen{\copy\@brx\kern-0.5\wd\@brx\usebox{\@brx}}}
\newcommand{\rrangle}[1][]{\savebox{\@brx}{\(\m@th{#1\rangle}\)}%
  \mathclose{\copy\@brx\kern-0.5\wd\@brx\usebox{\@brx}}}

\newtheorem{theorem}[equation]{Theorem}

\newtheorem{lemma}[equation]{Lemma}
\newtheorem{proposition}[equation]{Proposition}

\newtheorem{question}[equation]{Question}

\newtheorem{definition}[equation]{Definition}

\newtheorem{remark}[equation]{Remark}

\numberwithin{equation}{section}

\usepackage[all,cmtip]{xy}

\newcommand{\Aut}{\operatorname{Aut}}

\newcommand{\Oo}{\operatorname{O}}

\newcommand{\GL}{\operatorname{GL}}

\newcommand{\id}{\mathrm{id}}

\usepackage{tikz}

\usetikzlibrary{shapes,shadows,arrows,trees}
\usepackage{epstopdf}
\usepackage{inputenc}

\begin{document}

\title[Representation and twisted conjugacy of virtual Artin groups]{Linear representations, crystallographic quotients, and twisted conjugacy of virtual Artin groups}

\author{Neeraj Kumar Dhanwani}
\email{neerajk.dhanwani@gmail.com}
\address{Department of Mathematical Sciences, Indian Institute of Science Education and Research (IISER) Mohali, Sector 81,  S. A. S. Nagar, P. O. Manauli, Punjab 140306, India.}

\author{Pravin Kumar}
\email{pravin444enaj@gmail.com}
\address{Department of Mathematical Sciences, Indian Institute of Science Education and Research (IISER) Mohali, Sector 81,  S. A. S. Nagar, P. O. Manauli, Punjab 140306, India.}

\author{Tushar Kanta Naik}
\email{tushar@niser.ac.in}
\address{School of Mathematical Sciences, National Institute of Science Education and Research, Bhubaneswar, An OCC of Homi Bhabha National Institute, P. O. Jatni, Khurda 752050, Odisha, India.}
\author{Mahender Singh}
\email{mahender@iisermohali.ac.in}
\address{Department of Mathematical Sciences, Indian Institute of Science Education and Research (IISER) Mohali, Sector 81,  S. A. S. Nagar, P. O. Manauli, Punjab 140306, India.}


\begin{abstract}
Virtual Artin groups were recently introduced by Bellingeri, Paris, and Thiel as broad generalizations of the well-known virtual braid groups. For each Coxeter graph $\Gamma$, they defined the virtual Artin group $VA[\Gamma]$, which is generated by the corresponding Artin group $A[\Gamma]$ and the Coxeter group $W[\Gamma]$, subject to certain mixed relations inspired by the action of $W[\Gamma]$ on its root system $\Phi[\Gamma]$. There is a natural surjection  $ \mathrm{VA}[\Gamma] \rightarrow W[\Gamma]$, with the kernel $PVA[\Gamma]$ representing the pure virtual Artin group. In this paper, we explore linear representations, crystallographic quotients, and twisted conjugacy of virtual Artin groups.  Inspired from the work of Cohen, Wales, and Krammer, we construct a linear representation of the virtual Artin group $VA[\Gamma]$. As a consequence of this representation, we deduce that if $W[\Gamma]$ is a spherical Coxeter group, then $VA[\Gamma]/PVA[\Gamma]'$ is a crystallographic group of dimension $ |\Phi[\Gamma]|$ with the holonomy group $W[\Gamma]$. We also classify the torsion elements in $VA[\Gamma]/PVA[\Gamma]'$ and determine precisely when two elements are conjugate in this group. Further, we investigate twisted conjugacy, and prove that each right-angled virtual Artin group admit the $R_\infty$-property.
\end{abstract}
\maketitle

\section{Introduction}
Braid groups are classical examples of Artin groups that appear in  many areas of mathematics, including combinatorial group theory, geometric group theory, knot theory, mapping class groups of surfaces, and quantum algebra, to name a few. A far-reaching generalization of the classical knot theory is the virtual knot theory, introduced by Kauffman in \cite{MR1721925}. Later, in \cite{MR2128049}, virtual braid groups were introduced, which play the role of braid groups in the Alexander-Markov correspondence between the set of Markov equivalence classes of virtual braids and the set of stable isotopy classes of virtual links. It is known due to Kamada \cite{MR2351010} that the virtual braid group $VB_n$ on $n \ge 2$ strands admits a presentation with generating set $\{\sigma_1,\dots,\sigma_{n-1},\tau_1,\dots,\tau_{n-1} \}$ and defining relations:

\begin{enumerate}
\item $\sigma_i\sigma_j=\sigma_j\sigma_i$ for $|i-j|\ge2$ and $\sigma_i\sigma_j\sigma_i=\sigma_j\sigma_i\sigma_j$ for $|i-j|=1$,
\item $\tau_i\tau_j=\tau_j\tau_i$ for $|i-j|\ge2$, \ $\tau_i\tau_j\tau_i=\tau_j\tau_i\tau_j$ for $|i-j|=1$ and $\tau_i^2=1$ for $1\le i\le n-1$,
\item $\tau_j\sigma_i=\sigma_i\tau_j$ for $|i-j|\ge2$  and $\tau_i\tau_j\sigma_i=\sigma_j\tau_i\tau_j$ for $|i-j|=1$.
\end{enumerate}

Recently, in \cite{MR4753773}, Bellingeri, Paris, and Thiel introduced a far-reaching generalization of the virtual braid group, one corresponding to each Coxeter group. Recall that, a {\it Coxeter matrix} on a countable set $S$ is a symmetric  matrix $(m_{s,t})_{s,t\in S}$, where $m_{s,t} \in \mathbb{N} \cup\{\infty\}$ and satisfy $m_{s,s}=1$ for all $s\in S$ and $m_{s,t}=m_{t,s}\ge2$ for all $s,t\in S$ with $s\neq t$. Each such matrix can be represented by its {\it Coxeter graph} $\Gamma$, which is a labeled simple graph with $S$ as its set of vertices and two vertices $s,t \in S$ are joined by an edge if $m_{s,t}\ge3$, and such an edge is labeled with $m_{s,t}$ if $m_{s,t}\ge 4$. Throughout this paper, we shall assume that  $S$ is finite.
\par

Given elements $x,y$ of a group $G$ and an integer $m\ge1$,  we denote by $\text{Prod}_{R}(x,y,m)$ the word $\cdots yxy$ of length $m$.  Let $\Gamma$ be a Coxeter graph and $(m_{s,t})_{s,t\in S}$ be its Coxeter matrix.  Then, the {\it Artin group} $A[\Gamma]$ associated with $\Gamma$ is the group defined by the following presentation: 
$$
A[\Gamma]=\Bigg\langle S ~\mid~ \text{Prod}_{R}(t,s,m_{s,t})=\text{Prod}_{R}(s,t,m_{s,t})~\text{for all}~s,t\in S~\text{with}~s\neq t~\text{and}~m_{s,t}<\infty \Bigg\rangle.
$$
Further, the {\it Coxeter group} $W[\Gamma]$ associated with $\Gamma$ is the quotient of $A[\Gamma]$ defined by the following presentation: 
\begin{eqnarray*}
W[\Gamma] &=& \Bigg\langle S ~\mid~\text{Prod}_{R}(t,s,m_{s,t})=\text{Prod}_{R}(s,t,m_{s,t})~\text{for all}~s,t\in S~\text{with}~s\neq t~\text{and}~m_{s,t}<\infty, \\
&&  s^2=1~\text{for all}~s\in S~ \Bigg\rangle.
\end{eqnarray*}
Observe that the standard presentation of the virtual braid group $VB_n$ combines the standard presentation of the Artin braid group $B_n$ with that of the  symmetric group $S_n$, along with mixed relations that imitate the action of $S_n$ on its root system. Using this observation, Bellingeri, Paris, and Thiel introduced in \cite{MR4753773} the virtual Artin group $VA[\Gamma]$ associated with a Coxeter graph $\Gamma$ as follows:

\begin{definition}\label{def_virtual_artin_group}
Let $\Gamma$ be a Coxeter graph  with Coxeter matrix $(m_{s,t})_{s,t\in S}$. Let $\mathcal{S}=\{\sigma_s\mid s\in S\}$ and $\mathcal{T}=\{\tau_s\mid s\in S\}$ be two sets that are in one-to-one correspondence with $S$. The {\it virtual Artin group} $\text{VA}[\Gamma]$ associated with $\Gamma$ is the group defined by the presentation with generating set $\mathcal{S}\sqcup \mathcal{T}$ and defining relations:
\begin{enumerate}
\item $\text{Prod}_{R}(\sigma_t,\sigma_s,m_{s,t})=\text{Prod}_{R}(\sigma_s,\sigma_t,m_{s,t})$ for all $s,t\in S$ with $s\neq t$ and $m_{s,t} < \infty$,
\item $\text{Prod}_{R}(\tau_t,\tau_s,m_{s,t})=\text{Prod}_{R}(\tau_s,\tau_t,m_{s,t})$ for all $s,t\in S$ with $s\neq t$ and $m_{s,t} < \infty$; 
\item[] $\tau_s^2=1$ for $s\in S$,
\item $\text{Prod}_{R}(\tau_s,\tau_t,m_{s,t}-1)\sigma_s=\sigma_r\text{Prod}_{R}(\tau_s,\tau_t,m_{s,t}-1)$, where $r=s $ if $m_{s,t}$ is even and $r=t $ if $m_{s,t}$ is odd, for all $s,t\in S$ with $s\neq t$ and $m_{s,t} < \infty$.
\end{enumerate}
\end{definition}

Let $\Gamma$ be a Coxeter graph  with Coxeter matrix $(m_{s,t})_{s,t\in S}$. There is a surjective group homomorphism $\pi_P: \mathrm{VA}[\Gamma] \rightarrow W[\Gamma]$ given by 
$$\pi_P(\sigma_s) = \pi_P(\tau_s)= s$$ 
for all $s \in S$. The kernel $PVA[\Gamma]$  of $\pi_P$ is called the {\it pure virtual Artin group}.  The homomorphism $\pi_P$ admits a section $\iota_W: W[\Gamma] \rightarrow \mathrm{VA}[\Gamma]$ given by $\iota_W(s)=\tau_s$ for all $s \in S$. Hence, we have the semi-direct product decomposition 
\begin{equation}\label{decomposition of VA Gamma}
\mathrm{VA}[\Gamma]=\mathrm{PVA}[\Gamma] \rtimes W[\Gamma].
\end{equation}
 We denote the usual action of $W[\Gamma]$ on $PVA[\Gamma]$ by 
\begin{equation}\label{conjugation action of coxeter group}
w\cdot g=\iota_W(w)\,g\,\iota_W(w)^{-1}
\end{equation}
 for $w\in W[\Gamma]$ and $g\in PVA[\Gamma]$.
\par

Let $\Pi=\{\alpha_s\mid s\in S\}$ be a set in one-to-one correspondence with $S$, whose elements we refer to  as {\it simple roots}. Let $V$ be the real vector space having $\Pi$ as its basis. Let $\langle\cdot,\cdot\rangle:V\times V\to \mathbb{R}$ be the symmetric bilinear form defined on the basis by
$$
\langle \alpha_s,\alpha_t\rangle=\left\{\begin{array}{ll}
-2\cos(\pi/m_{s,t})&\text{if}~m_{s,t}<\infty,\\
-2&\text{if}~m_{s,t}=\infty.
\end{array}\right.
$$
This gives a faithful linear representation $$\rho: W[\Gamma]\hookrightarrow\GL(V),$$ called the \emph{Tits representation} of $W[\Gamma]$, defined by
$$
\rho(s)(v)=v-\langle v,\alpha_s\rangle\alpha_s
$$
for $v\in V$ and $s\in S$. A direct check shows that 
$$\langle \beta, \gamma \rangle = \langle \rho(w)(\beta), \rho(w)(\gamma) \rangle$$
 for all $\beta, \gamma \in V$ and $w \in W[\Gamma]$. The set $\Phi[\Gamma]=\{\rho(w)(\alpha_s)\mid s\in S\text{ and }w\in W[\Gamma]\}$ is called the {\it root system} of $\Gamma$.  It is known due to Deodhar \cite{MR0647210} that $\Phi[\Gamma]$ is finite if and only if $W[\Gamma]$ is finite. It is important to note that Definition \ref{def_virtual_artin_group} integrates the standard presentation of the Artin group $A[\Gamma]$, the standard presentation of the Coxeter group $W[\Gamma]$, and some mixed relations that mimic the action of $W[\Gamma]$ on its root system $\Phi[\Gamma]$.  We say that $\Gamma$ or $A[\Gamma]$ or $W[\Gamma]$ or $VA[\Gamma]$ is {\it spherical} if the Coxeter group  $W[\Gamma]$ is finite. Further, it is called  {\it right-angled} if for each $s \ne t$, we have $m_{s, t}=2$ or $\infty$. 
\par

In this paper, we explore linear representations, crystallographic quotients, and twisted conjugacy of virtual Artin groups. The layout of the paper is as follows. In Section \ref{section linear representation}, inspired from the work of Cohen and Wales \cite{MR1942303} and Krammer \cite{MR1888796}, we construct a linear representation of the virtual Artin group $VA[\Gamma]$ such that its kernel contains the commutator subgroup $PVA[\Gamma]'$ of $PVA[\Gamma]$ (Theorem \ref{linear rep VAG}). Using this result, we further prove that  the group $PVA[\Gamma]/PVA[\Gamma]'$ is a free abelian group of rank $|\Phi[\Gamma]|$ (Proposition \ref{prop:freeabelian}). In Section \ref{sec crystallographic quotient}, we analyse crystallographic quotients of $VA[\Gamma]$. We prove that if $W[\Gamma]$ is a spherical Coxeter group with root system $\Phi[\Gamma]$, then $VA[\Gamma]/PVA[\Gamma]'$ is a crystallographic group of dimension $ |\Phi[\Gamma]|$ with the holonomy group $W[\Gamma]$ (Theorem \ref{thm crystallographic SVAG}). We also classify the torsion elements in $VA[\Gamma]/PVA[\Gamma]'$ (Theorem \ref{crystallographic torsion}) and determine precisely when two elements are conjugate in this group (Theorem \ref{crystallographic conjugacy}). Finally, in Section \ref{section twisted conjugacy}, we investigate twisted conjugacy, and prove that each right-angled virtual Artin group admit the $R_\infty$-property (Theorem \ref{R infinity theorem}). 
\medskip

\section{Linear representation of $VA[\Gamma]$}\label{section linear representation}
In this section, we construct a linear representation of the virtual Artin group $VA[\Gamma]$, drawing inspiration from the work of Cohen and Wales \cite{MR1942303}, which in turn was influenced by Krammer \cite{MR1888796}. This representation will enable us to determine the exact rank of  $PVA[\Gamma]/PVA[\Gamma]'$. Later, in Section \ref{sec crystallographic quotient}, this will further help us prove that $VA[\Gamma]/PVA[\Gamma]'$ is crystallographic of dimension the rank of  $PVA[\Gamma]/PVA[\Gamma]'$ for the spherical case.

\subsection{Presentation of $PVA[\Gamma]$}
Before we proceed, let us revisit the presentation of $PVA[\Gamma]$ from \cite{MR4753773}, which we will be using. Let $\text{VA}[\Gamma]$ be the virtual Artin group associated with a Coxeter graph $\Gamma$, and let $\Phi[\Gamma]$ be its root system. Further, denote by $\Phi^+[\Gamma]$ (respectively $\Phi^-[\Gamma]$) the set of elements $\beta\in\Phi[\Gamma]$ that can be written as $\beta=\sum_{s\in S}\lambda_s\alpha_s$ with $\lambda_s\ge0$ (respectively $\lambda_s\le0$) for all $s\in S$. It is well-known that $\Phi[\Gamma]=\Phi^+[\Gamma]\sqcup\Phi^-[\Gamma]$. In fact, $\Phi^-[\Gamma]=\{-\beta\mid\beta\in\Phi^+[\Gamma]\}$. The elements of $\Phi^+[\Gamma]$ are called {\it positive roots}, whereas that of $\Phi^-[\Gamma]$ are called {\it negative roots}.
\par

It has been proved in \cite{MR4753773} that the pure virtual Artin group $PVA[\Gamma]$ has a generating set labelled by the elements of the root system $\Phi[\Gamma]$ of $W[\Gamma]$, and the (left) action of $W[\Gamma]$ on $PVA[\Gamma]$ is determined by the action of $W[\Gamma]$ on its root system. We elaborate on this for later use.
\par

For each $\beta\in\Phi[\Gamma]$, let $r_\beta:V\to V$ be the linear reflection defined by $$r_\beta(v)=v-\langle v,\beta\rangle\beta$$
for all $v \in V$. By definition, we have $r_{\alpha_s}=\rho(s)$ and
$r_\beta=r_{-\beta}$ for all $s\in S$ and $\beta\in\Phi[\Gamma]$.  If $s\in S$ and $w\in W[\Gamma]$ are such that $\beta= \rho(w)(\alpha_s)$, then $r_\beta= \rho(wsw^{-1})$. Since $\rho$ is faithful, we sometimes identify $r_\beta$ as an element of $W[\Gamma]$ for each $\beta\in\Phi[\Gamma]$. 
\par

Define a new Coxeter matrix $\widehat M=(\widehat m_{\beta,\gamma})_{\beta,\gamma\in\Phi[\Gamma]}$ as follows:
\begin{itemize}
\item[(a)] $\widehat m_{\beta,\beta}=1$ for all $\beta\in\Phi[\Gamma]$.
\item[(b)] Let $\beta,\gamma\in\Phi[\Gamma]$ such that $\beta\neq\gamma$. If there exist $w\in W[\Gamma]$ and $s,t\in S$ such that $\beta= \rho(w)(\alpha_s)$, $\gamma=\rho(w)(\alpha_t)$ and $m_{s,t}<\infty$, then $\widehat m_{\beta,\gamma}=m_{s,t}$.
\item[(c)]  $\widehat m_{\beta,\gamma}=\infty$ otherwise.
\end{itemize}
Note that, in (b),  the definition of $\widehat m_{\beta,\gamma}$ does not depend on the choice of $w$, $s$ and $t$. For $\beta\in\Phi[\Gamma]$, if $s\in S$ and $w\in W[\Gamma]$ are such that $\beta=\rho(w)(\alpha_s)$, then we set $$\zeta_\beta :=w\cdot(\tau_s\sigma_s)\in PVA[\Gamma].$$ Let $\beta,\gamma\in\Phi[\Gamma]$ such that $m:=\widehat m_{\beta,\gamma}<\infty$. We define roots $\beta_1,\dots,\beta_m\in\Phi[\Gamma]$ by setting $\beta_1=\beta$ and
$$
\beta_k=\left\{\begin{array}{ll}
{\rm Prod}_R(r_\gamma,r_\beta,k-1)(\gamma)&\text{if}~k~\text{is even},\\
{\rm Prod}_R(r_\beta,r_\gamma,k-1)(\beta)&\text{if}~k~\text{is odd},
\end{array}\right.
$$
for $k\ge2$. We set 
\begin{equation}\label{eq:relPVA}
Z(\gamma,\beta,\widehat m_{\beta,\gamma}) :=\zeta_{\beta_m}\cdots\zeta_{\beta_2}\zeta_{\beta_1},
\end{equation}
and note that $Z(\beta,\gamma,\widehat m_{\beta,\gamma})=\zeta_{\beta_1}\zeta_{\beta_2}\cdots\zeta_{\beta_m}$, the reverse word of $Z(\gamma,\beta,\widehat m_{\beta,\gamma})$. We will need the following result  \cite[Theorem 2.6]{MR4753773}.

\begin{theorem}\label{prop:presentationpva}
The group $PVA[\Gamma]$ has a presentation with generators $\{ \zeta_\beta  \mid \beta\in\Phi[\Gamma] \}$ and defining relations
$$
Z(\gamma,\beta,\widehat m_{\beta,\gamma})=Z(\beta,\gamma,\widehat m_{\beta,\gamma})
$$
for $\beta,\gamma\in\Phi[\Gamma]$ with $\beta\neq\gamma$ and $\widehat m_{\beta,\gamma}<\infty$.
\end{theorem}

\begin{remark}\label{rmk:action}
For each $w\in W[\Gamma]$ and $\zeta_\beta\in PVA[\Gamma]$, a direct check shows that  $$w\cdot \zeta_\beta=\zeta_{\rho(w)(\beta)}.$$
\end{remark}
\medskip

\subsection{Construction of linear representation of $VA[\Gamma]$} 
In this subsection, we construct a linear representation of the virtual Artin group $VA[\Gamma]$. Consider the ring $R=\mathbb Z [x^{\pm 1}, y^{\pm 1}]$ of Laurent polynomials in the commuting variables $x$ and $y$. Let $U$ be the free $R$-module with basis $\{e_\beta  \mid \beta \in \Phi^{+}[\Gamma] \}$. For any $\beta \in \Phi^{+}[\Gamma]$ and $s \in S$, we define

$$
\psi\left(\sigma_s\right)\left(e_\beta\right)= \begin{cases}
x e_\beta  & \text{ if } \beta=\alpha_s, \\ 
e_{\rho(s)(\beta)} & \text{ otherwise,}
 \end{cases}
$$
and
$$
\psi\left(\tau_{s}\right)\left(e_\beta\right)= \begin{cases}
e_{\beta} & \text{ if } \beta=\alpha_s,  \\ 
y^{\epsilon(s,\beta)} e_{\rho(s)(\beta)}  & \text{ if } \text{ otherwise,} \end{cases}
$$
where $$\epsilon(s,\beta)= \begin{cases}
+1  & \text{ if } \langle \beta , \alpha_s \rangle < 0,  \\ 
-1 & \text{ if } \langle \beta , \alpha_s \rangle > 0, \\
0  & \text{ if } \langle \beta , \alpha_s \rangle =0
 .\end{cases}
$$ 

For any $\beta \in \Phi^{+}[\Gamma]$ and $s \in S$, a direct check shows that the root $\rho(s)(\beta)$ is positive unless $\beta=\alpha_s$. Thus, each $\psi(\sigma_s)$ and $\psi(\tau_s)$ is a well-defined $R$-linear automorphism of $U$. Recall the notation $\mathcal{S} =\{\sigma_s \mid s \in S\}$ and $\mathcal{T}= \{\tau_s \mid s \in S\}$. Thus, we have a group homomorphism  $\psi$ from the free group $F(\mathcal{S} \cup \mathcal{T})$ on $\mathcal{S} \cup \mathcal{T}$ to the group $\GL(U)$ of $R$-linear automorphisms of $U$.

\begin{remark}\label{rmk:pi(w)}
For each $w \in W[\Gamma]$, let 
$$\Pi(w)= \{\beta \in \Phi^+[\Gamma] \mid  \rho(w)(\beta) \in \Phi^-[\Gamma] \}.$$
Let $w=s_1 \cdots s_r$ be a  reduced expression of $w$ and 
$$
\beta^w_i:=\rho(s_r s_{r-1} \cdots s_{i+1})\left(\alpha_{s_i}\right)~\textrm{for each}~1 \le i \le r-1~\textrm{and}~ \beta^w_r:=\alpha_{s_r}.
$$
 Then,  $\beta^w_i$'s are distinct and $\Pi(w)=\left\{\beta^w_1, \ldots, \beta^w_r\right\}$. Consequently, $|\Pi(w)|= \ell(w)=r$, where $\ell(w)$ is the length of $w$.
\end{remark}

For each reduced word $\widehat{w}= \tau_{s_1}^{\varepsilon_1}\tau_{s_2}^{\varepsilon_2}\cdots \tau_{s_r}^{\varepsilon_r}$ in the free group $F(\mathcal{T}$) on the set $\mathcal{T}$, where $\varepsilon_i\in \{1,-1\}$, let $w=s_1^{\varepsilon_1}s_2^{\varepsilon_2}\cdots s_r^{\varepsilon_r}$ be the corresponding element of $W[\Gamma]$. Then, for $\beta \in \Phi^{+}[\Gamma]$, we define $\kappa(\widehat{w}, \beta) \in \mathbb Z$ such that   
\begin{equation}\label{kappa def}
 \psi(\widehat{w})(e_\beta)=\begin{cases} y^{\kappa(\widehat{w}, \beta)} e_{-\rho(w)(\beta)}  & \text{ if } \beta\in \Pi(w), \\ 
y^{\kappa(\widehat{w}, \beta)} e_{\rho(w)(\beta)} & \text{ otherwise}.
\end{cases}
\end{equation}
Note that, if $\widehat{w}=\tau_s$ and $\beta \neq \alpha_s$, then  $\kappa(\widehat{w}, \beta)=\epsilon(s,\beta)$.
\par

\begin{lemma}\label{lemma on representation}
The following assertions hold:
\begin{enumerate}
\item If $\beta \in \Phi^{+}[\Gamma]$ and $s \in S$, then $\epsilon(s,\beta)=-\epsilon (s, \rho(s)(\beta))$.
\item Let $\widehat{w}=\tau_{s_1}\tau_{s_2}\cdots \tau_{s_r}$ be an element of $F(\mathcal{T}$), $\widetilde{w}=\tau_{s_r}\tau_{s_{r-1}}\cdots \tau_{s_1}$ its reverse and $w=s_1 s_2 \cdots s_r$ the corresponding element of $W[\Gamma]$. If $\beta \not\in \Pi(w)$ and $\gamma =\rho(w)(\beta)$, then $\kappa(\widehat{w}, \beta)=-\kappa(\widetilde{w}, \gamma)$.
\end{enumerate}
\end{lemma}

\begin{proof}
For  $\beta \in \Phi^{+}[\Gamma]$ and $s \in S$, we have $\langle \beta , \alpha_s \rangle = \langle \rho(s)(\beta) , \rho(s)(\alpha_s) \rangle= -\langle \rho(s)(\beta) , \alpha_s \rangle$, and hence $\epsilon(s,\beta)=-\epsilon (s, \rho(s)(\beta))$.
\par

For the second assertion, since $\beta \not\in \Pi(w)$, we see that
\begin{eqnarray*}
&& \kappa(\widetilde{w}, \gamma)\\
 &=& \epsilon(s_1,\gamma) + \epsilon(s_{2},\rho(s_1)(\gamma))+\epsilon(s_{3},\rho(s_{2}s_1)(\gamma))+\cdots+\epsilon(s_r,\rho(s_{r-1}\cdots s_1)(\gamma))\\
&=& \epsilon(s_1,\rho(s_1)\,\rho(s_{2}\cdots s_{r})(\beta)) + \epsilon(s_{2},\rho(s_2)\,\rho(s_3\cdots s_r)(\beta))+\cdots +\epsilon(s_r,\rho(s_{r})(\beta))\\
&=& - \Bigg(\epsilon(s_1,\rho(s_{2}\cdots s_{r})(\beta)) + \epsilon(s_{2},\rho(s_3\cdots s_r)(\beta))+\epsilon(s_{3},\rho(s_{4}\cdots s_r)(\beta))+\cdots+\epsilon(s_r,\beta) \Bigg)\\
&& \textrm{(by assertion (1))}\\
&=&-\Bigg(\epsilon(s_r,\beta) + \epsilon(s_{r-1},\rho(s_r)(\beta))+\epsilon(s_{r-2},\rho(s_{r-1}s_r)(\beta))+\cdots+\epsilon(s_1,\rho(s_2\cdots s_r)(\beta))\Bigg)\\
&=&-\kappa(\widehat{w}, \beta),
\end{eqnarray*}
and hence $\kappa(\widehat{w}, \beta)=-\kappa(\widetilde{w}, \gamma)$.
\end{proof}

\begin{theorem}\label{linear rep VAG}
Let $\text{VA}[\Gamma]$ be the virtual Artin group associated with a Coxeter graph $\Gamma$. Then, the map $\psi$ induces a representation $\Psi:VA[\Gamma]\to \GL(U)$ such that the commutator subgroup $PVA[\Gamma]'$ of $PVA[\Gamma]$ is contained in $\ker(\Psi)$.
\end{theorem}

\begin{proof}
To prove that  $\psi$ induces a representation, it remains to check that it preserves all the defining relations of $VA[\Gamma]$.
 \begin{enumerate}
 \item Consider the relation $\operatorname{Prod}_R\left(\sigma_t, \sigma_s, m_{s, t}\right)=\operatorname{Prod}_R\left(\sigma_s, \sigma_t, m_{s, t}\right)$ in $VA[\Gamma]$ for fixed $s, t \in S$ with $s \neq t$ and $m_{s, t} < \infty$.  Let $w=\operatorname{Prod}_R\left(t, s, m_{s, t}\right)=\operatorname{Prod}_R\left(s,t, m_{s, t}\right) \in W[\Gamma]$. Let $\Pi(w)$ be the set of positive roots that are sent to negative roots by $\rho(w)$. For any $\beta \in \Phi^{+}[\Gamma]$, it is easy to check that 
\begin{align*}
\psi(\operatorname{Prod}_R\left(\sigma_t, \sigma_s, m_{s, t}\right))\left(e_\beta\right)&= \begin{cases}
x e_{-\rho(w)(\beta)}  & \text{ if } \beta\in \Pi(w), \\ 
e_{\rho(w)(\beta)} & \text{ otherwise,} 
 \end{cases}\\
 &=\psi(\operatorname{Prod}_R\left(\sigma_s, \sigma_t, m_{s, t}\right))\left(e_\beta\right),
\end{align*}
which is desired.
\item For any $\beta \in \Phi^{+}[\Gamma]$ and $s \in S$, by Lemma \ref{lemma on representation}(1), we have $\epsilon(s,\beta)=-\epsilon (s, \rho(s)(\beta))$. Hence, $\psi(\tau_s^2) = \id$ for all $s \in S$.

\item Next, we consider the relation $\operatorname{Prod}_R\left(\tau_t, \tau_s, m_{s, t}\right)=\operatorname{Prod}_R\left(\tau_s, \tau_t, m_{s, t}\right)$ in $VA[\Gamma]$ for  fixed $s, t \in S$ with $s \neq t$ and $m_{s, t} < \infty$.  Set $m=m_{s, t}$ for convenience. Let $\widehat{w_1}=\operatorname{Prod}_R\left(\tau_t, \tau_s, m \right)$ and $\widehat{w_2}= \operatorname{Prod}_R\left(\tau_s, \tau_t, m\right)$ considered as elements of $F(\mathcal{S} \cup \mathcal{T})$. Further, let $w_1=\operatorname{Prod}_R\left(t,s, m\right)$ and $w_2=\operatorname{Prod}_R\left(s, t, m\right)$ denote  two reduced expressions of the corresponding element in $W[\Gamma]$. Since $\rho(w_1)=\rho(w_2)$, we have $\Pi(w_1)= \Pi(w_2)$. On the other hand,  by Remark~\ref{rmk:pi(w)}, we have $\Pi(w_1)=\{\beta^{w_1}_1, \ldots, \beta^{w_1}_m\}$ and $\Pi(w_2)=\{\beta^{w_2}_1, \ldots, \beta^{w_2}_m\}$. It is easy to verify that $\beta^{w_1}_i=\beta^{w_2}_{m-i+1}$ \cite[Proof of Lemma 2.5]{MR4753773}. In particular, $\beta^{w_1}_m=\alpha_s=\beta^{w_2}_1$ and $\beta^{w_1}_1=\alpha_t=\beta^{w_2}_m$. Let $\beta\in \Phi^{+}[\Gamma]$ be fixed.  For each $0\leq i \leq m-1$, set 
$$\nu^{\beta}_i :=\rho(\operatorname{Prod}_R\left(t,s, i\right))(\beta)~\textrm{and}~ \mu^{\beta}_i:=\rho(\operatorname{Prod}_R\left(s,t, i\right))(\beta).$$ 
Then, we have the following sub-cases:
  \item[Case (3a)] Suppose that $\beta=\beta^{w_1}_i\in \Pi(w_1)$ for some $1 \le i\leq \lceil m/2\rceil $. We write $\widehat{w_1}= \widehat{u_1}  \widehat{v_1}$, where  $\widehat{u_1}$ is the subword of $\widehat{w_1}$ consisting of the first $2i-1$ letters from the left and $\widehat{v_1}=\operatorname{Prod}_R\left(\tau_t, \tau_s, m-2i+1\right)$ is the subword of $\widehat{w_1}$ consisting of the first $m-2i+1$ letters from the right. Further, we can write $\widehat{u_1}=\widehat{x_1}\tau_s \widetilde{x_1} $ or $\widehat{x_1}\tau_t \widetilde{x_1}$,  where $\widehat{x_1}$ is the subword of $\widehat{u_1}$ consisting of the first $i-1$ letters from the left and $\widetilde{x_1}$ is its reverse word. Without loss of generality, we can assume that $\widehat{u_1}=\widehat{x_1}\tau_s \widetilde{x_1}$. Let $w_1=u_1v_1=x_1 s x_1^{-1}v_1$ be the corresponding element in $W[\Gamma]$. Note that $\Pi(v_1)= \{ \beta^{w_1}_{2i}, \ldots, \beta^{w_1}_m\} \subset \Pi(w_1)$.
Hence, it follows that the root $\rho(v_1)(\beta^{w_1}_i)$ is positive. This gives
  \begin{eqnarray*}
   \psi(\widehat{w_1})(e_{\beta}) &=&    \psi(\widehat{u_1} \widehat{v_1})(e_{\beta^{w_1}_i})\\
&=& \psi(\widehat{u_1})(y^{\kappa(\widehat{v_1},\beta^{w_1}_i)}e_{\rho(v_1)(\beta^{w_1}_i)})\\
&=&  y^{\kappa(\widehat{x_1},-\rho(sx_1^{-1}v_1)(\beta^{w_1}_i))} y^{\kappa(\widetilde{x_1},\rho(v_1)(\beta^{w_1}_i))} y^{\kappa(\widehat{v_1},\beta^{w_1}_i)}e_{-\rho(w_1)(\beta^{w_1}_i)}\\
&=& y^{\kappa(\widehat{v_1},\beta^{w_1}_i)}e_{-\rho(w_1)(\beta^{w_1}_i)},
\end{eqnarray*}
where the fourth equality follows from Lemma \ref{lemma on representation}(2) and the fact that $-\rho(sx_1^{-1}v_1)(\beta^{w_1}_i)=\rho(x_1^{-1})\rho(v_1)(\beta^{w_1}_i)=\alpha_s$.
\par

 Similarly, we write $\widehat{w_2}= \widehat{u_2} \widehat{v_2}$, where $\widehat{u_2}$ is the subword of $\widehat{w_2}$ consisting of the first $m-2i+1$ letters from the left and $\widehat{v_2}$ is the subword of $\widehat{w_2}$ consisting of the first $2i-1$ letters from the right.  Since $2i-1$ is odd, the first letter of  $\widehat{u_2}$ from the right is $s$, and hence $\widehat{u_2}=\operatorname{Prod}_R\left(\tau_t, \tau_s, m-2i+1\right)=\widehat{v_1}$. Let $w_2=u_2v_2$ be the corresponding element in $W[\Gamma]$. It is easy to verify that $\rho(v_2)(\beta^{w_1}_i)=-\beta^{w_1}_i=\rho(v_2)(\beta^{w_2}_{m-i+1})$. Thus, we have
\begin{align*}
   \psi(\widehat{w_2})(e_{\beta})&=\psi(\widehat{u_2} \widehat{v_2})(e_{\beta^{w_1}_i})\\
   &= \psi(\widehat{u_2})(e_{-\rho(v_2)(\beta^{w_1}_i)}), \quad \textrm{by argument similar to that of $\psi(\widehat{w_1})(e_\beta)$}\\
   &= y^{\kappa(\widehat{u_2},-\rho(v_2)(\beta^{w_1}_i))}e_{-\rho(w_2)(\beta^{w_1}_i)} \\
 &=y^{\kappa(\widehat{v_1},\beta^{w_1}_i)}e_{-\rho(w_1)(\beta^{w_1}_i)}, \quad \textrm{since $w_1=w_2$ in $W[\Gamma]$}\\
 &=    \psi(\widehat{w_1})(e_{\beta}).   
\end{align*}
\par
\item[Case (3b)]   Suppose that $\beta=\beta^{w_1}_i\in \Pi(w_1)$ for some $\lceil m/2 \rceil \leq i \le m$. In this case, $1 \le m-i+1 \leq \lceil m/2\rceil$. The proof is similar to Case (3a) and follows by interchanging the roles of $\widehat{w_1}$ and $\widehat{w_2}$.
\par

\item[Case (3c)] Suppose that $\beta \not\in \Pi(w_1) =\Pi(w_2)$. Let $\psi(\widehat{w_1})(e_{\beta})=y^{\kappa(\widehat{w_1}, \beta)} e_{\rho(w_1)(\beta)}$ and $\psi(\widehat{w_2})(e_{\beta})=y^{\kappa(\widehat{w_2}, \beta)} e_{\rho(w_2)(\beta)}$. Then, we have
\begin{eqnarray*}
\kappa(\widehat{w_1}, \beta) &=& \epsilon(s,\beta) + \epsilon(t,\rho(s)(\beta))+\epsilon(s,\rho(ts)(\beta))+\epsilon(t,\rho(sts)(\beta))+\cdots\\
&=& \epsilon(s,\nu^\beta_0)+ \epsilon(t,\nu^\beta_1)+\epsilon(s,\nu^\beta_2)+\epsilon(t,\nu^\beta_3)+\cdots +\epsilon(r_1,\nu^\beta_{m-1}),
\end{eqnarray*}
where
$$r_1=\begin{cases} t &\text{ if } m \text{ is even, }\\
s &\text{ if } m \text{ is odd. }
 \end{cases}$$
Similarly,
\begin{eqnarray*}
\kappa(\widehat{w_2}, \beta) &=& \epsilon(t,\beta) + \epsilon(s,\rho(t)(\beta))+\epsilon(t,\rho(st)(\beta))+\epsilon(s,\rho(tst)(\beta))+\cdots\\
&=& \epsilon(t,\mu^\beta_0)+ \epsilon(s,\mu^\beta_1)+\epsilon(t,\mu^\beta_2)+\epsilon(s,\mu^\beta_3)+\cdots +\epsilon(r_2,\mu^\beta_{m-1}),
\end{eqnarray*}
where
$$r_2=\begin{cases} s &\text{ if } m \text{ is even,}\\
t &\text{ if } m \text{ is odd. } 
\end{cases}$$
It is easy to check the following equality \cite[p.197, Claim 1(3)]{MR4753773}
\begin{equation}\label{prod st m-1action on s}
\rho(\text{Prod}_{R}(s,t,m-1))(\alpha_s)=\begin{cases}
\alpha_s & \text{ if } m \text{ is even,}\\
\alpha_t & \text{ if } m \text{ is odd.}
\end{cases}
\end{equation}

Suppose that $m$ is even. We claim that $\epsilon(s,\nu^\beta_k)= \epsilon(s,\mu^\beta_{m-1-k})$ and $\epsilon(t,\nu^\beta_j)= \epsilon(t,\mu^\beta_{m-1-j})$ for $k$ even and $j$ odd. Indeed, using \eqref{prod st m-1action on s} and invariance of the bilinear form, we have
\begin{eqnarray*}
\langle \nu^\beta_k, \alpha_s \rangle &=& \langle \rho(\operatorname{Prod}_R\left(t,s, k\right))(\beta) , \alpha_s \rangle\\
&=& \langle \rho(\text{Prod}_{R}(s,t,m-1)~ \operatorname{Prod}_R\left(t,s, k\right))(\beta), \alpha_s \rangle\\
&=& \langle \rho(\operatorname{Prod}_R\left(s,t, m-1-k\right))(\beta) , \alpha_s \rangle\\
&=& \langle \mu^\beta_{m-1-k}, \alpha_s \rangle
\end{eqnarray*}
and
\begin{eqnarray*}
\langle \nu^\beta_{j}, \alpha_t \rangle &=&
\langle \rho(\operatorname{Prod}_R\left(t,s, j\right))(\beta) , \alpha_t \rangle\\ 
&=& \langle \rho(\text{Prod}_{R}(t, s,m-1)~ \operatorname{Prod}_R\left(t,s, j\right))(\beta), \alpha_t \rangle\\
&=& \langle \rho(\operatorname{Prod}_R\left(s,t, m-1-j\right))(\beta) , \alpha_t \rangle\\
&=& \langle \mu^\beta_{m-1-j}, \alpha_t \rangle.
\end{eqnarray*}
This gives $\epsilon(s,\nu^\beta_k)= \epsilon(s,\mu^\beta_{m-1-k})$ and $\epsilon(t,\nu^\beta_j)= \epsilon(t,\mu^\beta_{m-1-j})$, and hence $\kappa(\widehat{w_1}, \beta)=\kappa(\widehat{w_2}, \beta)$.
\par
Now, suppose that $m$ is odd. In this case, we claim that $\epsilon(s,\nu^\beta_k)= \epsilon(t,\mu^\beta_{m-1-k})$ and $\epsilon(t,\nu^\beta_j)= \epsilon(s,\mu^\beta_{m-1-j})$ for $k$ even and $j$ odd. As before, we have
\begin{eqnarray*}
\langle \nu^\beta_{k}, \alpha_s \rangle &=&
\langle \rho(\operatorname{Prod}_R\left(t,s, k\right))(\beta) , \alpha_s \rangle\\
&=& \langle \rho(\text{Prod}_{R}(s, t,m-1)~ \operatorname{Prod}_R\left(t,s, k\right))(\beta), \alpha_t \rangle\\
&=& \langle \rho(\operatorname{Prod}_R\left(s,t, m-1-k\right))(\beta) , \alpha_t \rangle\\
&=& \langle \mu^\beta_{m-1-k}, \alpha_t \rangle
\end{eqnarray*}
and 
\begin{eqnarray*}
\langle \nu^\beta_{j}, \alpha_t \rangle &=&
\langle \rho(\operatorname{Prod}_R\left(t,s, j\right))(\beta) , \alpha_t \rangle \\ 
&=& \langle \rho(\text{Prod}_{R}(t,s,m-1)~ \operatorname{Prod}_R\left(t,s, j\right))(\beta), \alpha_t \rangle\\
&=& \langle \rho(\operatorname{Prod}_R\left(s,t, m-1-j\right))(\beta) , \alpha_s \rangle\\
&=& \langle \mu^\beta_{m-1-j}, \alpha_s \rangle.
\end{eqnarray*}
The preceding equations imply that $\epsilon(s,\nu^\beta_k)= \epsilon(t,\mu^\beta_{m-1-k})$ and $\epsilon(t,\nu^\beta_j)= \epsilon(s,\mu^\beta_{m-1-j})$, and hence  $\kappa(\widehat{w_1}, \beta)=\kappa(\widehat{w_2}, \beta)$. 
\par

\item Finally, we consider the mixed relation $\text{Prod}_{R}(\tau_s,\tau_t,m_{s,t}-1)\sigma_s=\sigma_r\text{Prod}_{R}(\tau_s,\tau_t,m_{s,t}-1)$ for fixed $s,t\in S$ with $s\neq t$ and $m_{s,t} < \infty$, where $r=s $ if $m_{s,t}$ is even and $r=t $ if $m_{s,t}$ is odd. Set $m=m_{s, t}$ for convenience. Let $\widehat{w_1}=\operatorname{Prod}_R\left(\tau_t, \tau_s, m\right)$, $\widehat{w_2}= \operatorname{Prod}_R\left(\tau_s, \tau_t, m\right)$ and $\widehat{u}=\operatorname{Prod}_R\left(\tau_s,\tau_t, m-1\right)$. Then $\widehat{w_1}=\widehat{u}\; \tau_s$ and $\widehat{w_2}=\tau_r \; \widehat{u}$. Let $w_1=\operatorname{Prod}_R\left(t,s, m\right)$ and $w_2=\operatorname{Prod}_R\left(s, t, m\right)$ denote the corresponding element in $W[\Gamma]$. Then, we have
\begin{eqnarray*}
\psi( \text{Prod}_{R}(\tau_s,\tau_t,m-1)\sigma_s)(e_\beta) &=& \psi(\widehat{u}\sigma_s)(e_\beta)\\
&=& \begin{cases} y^{\kappa(\widehat{u}, \beta)} x e_{-\rho(w_1)(\beta)} & ~\text{if}~\beta= \beta_m^{w_1}=\alpha_s,\\
y^{\kappa(\widehat{u}, \rho(s)(\beta))}  e_{-\rho(w_1)(\beta)} &~\text{if}~\beta=\beta_i^{w_1}~\text{for}~i < m,\\
y^{\kappa(\widehat{u}, \rho(s)(\beta))}  e_{\rho(w_1)(\beta)} &~\text{if}~ \beta \not\in \Pi(w_1),
 \end{cases}
\end{eqnarray*}
and
\begin{eqnarray*}
\psi(\sigma_r\text{Prod}_{R}(\tau_s,\tau_t,m-1))(e_\beta) &=& \psi(\sigma_r\widehat{u})(e_\beta)\\
&= & \psi(\sigma_r)\left(\begin{cases}
y^{\kappa(\widehat{u}, \beta)} e_{\alpha_r} & \text{ if } \beta=\beta_m^{w_1}=\beta_1^{w_2}=\alpha_s,\\
y^{\kappa(\widehat{u}, \beta)}  e_{-\rho(u)(\beta)} & ~\text{if}~\beta=\beta_i^{w_1}=\beta_{m-i+1}^{w_2}~\text{for}~i < m,\\
y^{\kappa(\widehat{u}, \beta)}  e_{\rho(u)(\beta)} & ~\text{if}~ \beta \not\in \Pi(w_1)=\Pi(w_2),
\end{cases}\right)\\
&=& \begin{cases}
y^{\kappa(\widehat{u}, \beta)} x e_{-\rho(w_2)(\beta)} & \text{ if } \beta=\beta_m^{w_1}=\beta_1^{w_2}=\alpha_s,\\
y^{\kappa(\widehat{u}, \beta)}  e_{-\rho(w_2)(\beta)} & ~\text{if}~\beta=\beta_i^{w_1}=\beta_{m-i+1}^{w_2}~\text{for}~i < m,\\
y^{\kappa(\widehat{u}, \beta)}  e_{\rho(w_2)(\beta)}& ~\text{if}~ \beta \not\in \Pi(w_1)=\Pi(w_2).
\end{cases}
\end{eqnarray*}
Recalling notation from Case (3), for each $0\leq i \leq m$, let $\nu^\beta_i=\rho(\operatorname{Prod}_R\left(t,s, i\right))(\beta)$ and 
$\mu^\beta_i=\rho(\operatorname{Prod}_R\left(s,t, i\right))(\beta)$. Since $\kappa(\widehat{w_1}, \beta)= \kappa(\widehat{w_2}, \beta)$, it follows that
$$ \kappa(\widehat{u}, \rho(s)(\beta))= \kappa(\widehat{w_1}, \beta)-\epsilon(s,\beta)= \kappa(\widehat{w_1}, \beta)-\epsilon(s,\nu^\beta_0)=
\kappa(\widehat{w_2}, \beta)-\epsilon(r,\mu^\beta_{m-1})= \kappa(\widehat{u}, \beta).
$$
\end{enumerate}
Thus, $\psi$ induces a group homomorphism, say, $\Psi:VA[\Gamma]\to \GL(U)$.
\medskip

We now prove the second assertion. Since $\Psi$ is a group homomorphism, for each $w\in \langle \tau_s \mid s\in S \rangle\subset VA[\Gamma]$ and $\beta \in \Phi^{+}[\Gamma]$, the expression \eqref{kappa def} takes the form
\begin{equation}\label{kappa def revised}
\Psi(w)(e_\beta)=\begin{cases} y^{\kappa(w, \beta)} e_{-\rho(w)(\beta)}  & \text{ if } \beta\in \Pi(w), \\ 
y^{\kappa(w, \beta)} e_{\rho(w)(\beta)} & \text{ otherwise},\end{cases}
\end{equation}
where $\kappa(w, \beta) \in \mathbb Z$ is independent of the choice of the word representing $w$. Using the fact that $\Psi$ is a homomorphism, we get
\begin{equation}\label{kappa w inverse relation}
\kappa(w, \beta) = - \kappa(w^{-1}, \pm \rho(w)(\beta))
\end{equation} 
and
$$\Psi(\tau_s\sigma_s)(e_\beta)=\begin{cases}
xe_\beta & \text{ if } \beta=\alpha_s,\\
y^{\epsilon(s, \rho(s)(\beta))}e_\beta & \text{ otherwise, }
\end{cases} \quad \quad \text{and} \quad \quad \Psi(\sigma_s\tau_s)(e_\beta)=\begin{cases}
xe_\beta & \text{ if } \beta=\alpha_s,\\
y^{\epsilon(s, \beta)}e_\beta & \text{ otherwise. }
\end{cases}
$$
For simplicity of notation, for an element $w \in W[\Gamma]$, we denote its image $\iota_W(w)$ by $w$ itself. For each $\delta\in\Phi[\Gamma]$, let $\zeta_\delta=w\cdot(\tau_s\sigma_s)=\iota_W(w) (\tau_s\sigma_s)\iota_W(w^{-1})= w (\tau_s\sigma_s) w^{-1}\in PVA[\Gamma]$, where  $w\in W[\Gamma]$ and $s\in S$ are such that $\delta=\rho(w)(\alpha_s)$. Then, for each $\beta \in \Phi^{+}[\Gamma]$, using \eqref{kappa def revised} and \eqref{kappa w inverse relation}, we obtain
\begin{align*}
\Psi(\zeta_\delta)(e_\beta)
&= \Psi(w)\Psi(\tau_s\sigma_s) \Psi(w^{-1})(e_\beta) \\
&= \Psi(w)\Psi(\tau_s\sigma_s)(y^{\kappa(w^{-1},\beta)}e_{\pm \rho(w^{-1})(\beta)})\\
&=\Psi(w)\left(\begin{cases} 
y^{\kappa(w^{-1},\beta)} x e_{\pm \rho(w^{-1})(\beta)} & \text{ if } \pm \rho(w^{-1})(\beta)=\alpha_s,\\
&\\
y^{\kappa(w^{-1},\beta)} y^{\epsilon(s, \pm\rho(sw^{-1})(\beta))}e_{\pm \rho(w^{-1})(\beta)} & \text{ otherwise, }
\end{cases}\right)\\
&= \begin{cases} x y^{\kappa(w^{-1},\beta)} y^{\kappa(w,\pm \rho(w^{-1})(\beta))} e_{\beta} & \text{ if } \pm \rho(w^{-1})(\beta)=\alpha_s,\\
&\\
y^{\kappa(w^{-1},\beta)} y^{\epsilon(s, \pm\rho(sw^{-1})(\beta))} y^{\kappa(w,\pm \rho(w^{-1})(\beta))} e_{\beta} & \text{ otherwise, } 
\end{cases}\\
&= \begin{cases} 
x e_{\beta} & \text{ if } \pm \rho(w^{-1})(\beta)=\alpha_s,\\
 y^{\epsilon(s, \pm\rho(sw^{-1})(\beta))}  e_{\beta} & \text{ otherwise,}
 \end{cases}\\
 &= \begin{cases} 
x e_{\beta} & \text{ if }  \beta=\pm \delta,\\
 y^{\epsilon(s, \pm\rho(sw^{-1})(\beta))}  e_{\beta} & \text{ otherwise.}
 \end{cases}
\end{align*}
\par
Fixing an ordering on the basis $\{e_\beta \mid \beta \in \Phi^{+}[\Gamma]\}$ of $U$, we see that the matrix representation of $\Psi(\zeta_\delta)$ is the diagonal matrix. Since diagonal matrices commute, it follows from Theorem \ref{prop:presentationpva} that the commutator subgroup  $PVA[\Gamma]'$ of $PVA[\Gamma]$ is contained in $\ker(\Psi)$. This completes the proof of the theorem.
\end{proof}

\begin{remark}
The representation $\Psi:VA[\Gamma] \to \GL(U)$  is finite-dimensional if and only if $VA[\Gamma]$  is spherical. In view of \cite[Corollary 2.4]{MR4753773}, the Artin group $A[\Gamma]$ can be viewed as a subgroup of $VA[\Gamma]$. Further, the restriction of $\Psi$ to $A[\Gamma]$ agrees with the specialisation at $q=1$ of the Krammer's representation \cite{MR1888796} for Artin group of type $A_m$ for $m \geq 1$, and the representation given in \cite{MR1942303, MR2004479} for Artin group of type $E_6$, $E_7$, $E_8$ and $D_m$ for $m \geq 4$.
\end{remark}

Consider the split  exact sequence
\begin{equation*}
\begin{tikzcd}
1 \arrow[r] & PVA[\Gamma] \arrow[r] \arrow[r] & VA[\Gamma] \arrow[r, , "\pi_P"] & {W[\Gamma]} \arrow[r] \arrow[l,  "\iota", bend left] & 1
\end{tikzcd}
\end{equation*}
which, in turn, induces the split  exact sequence

\begin{equation}\label{splitting for crystallographic}
\begin{tikzcd}
1 \arrow[r] & {PVA[\Gamma]/PVA[\Gamma]'} \arrow[r] \arrow[r] & {VA[\Gamma]/PVA[\Gamma]'} \arrow[r, , "\overline{\pi}_P"] & {W[\Gamma]} \arrow[r] \arrow[l,  "\overline{\iota}", bend left] & 1.
\end{tikzcd}
\end{equation}

\begin{proposition}\label{prop:freeabelian}
Let $VA[\Gamma]$ be the virtual Artin group associated with a Coxeter graph $\Gamma$, and let $\Phi[\Gamma]$ be it root system. Then the group $PVA[\Gamma]/PVA[\Gamma]'$ is a free abelian group of rank $|\Phi[\Gamma]|$.
\end{proposition}

\begin{proof}
In view of Theorem \ref{prop:presentationpva}, the images of elements from the set $\{\zeta_\beta  \mid \beta\in\Phi[\Gamma]\}$ generate the quotient $PVA[\Gamma]/PVA[\Gamma]'$. Further, the defining relations stated in Theorem \ref{prop:presentationpva} become trivial as a consequence of commutation, and hence $PVA[\Gamma]/PVA[\Gamma]'$ is a free abelian group. 
\par
We claim that $\zeta_\beta\ne \zeta_\gamma \mod PVA[\Gamma]'$ for any $\beta \neq \gamma \in \Phi[\Gamma]$. Suppose that $\gamma \in \Phi[\Gamma]\setminus \{\pm \beta\}$. Then we see that  $\Psi(\zeta_{\beta})(e_{\pm \beta})=xe_{\pm \beta}$, whereas $\Psi(\zeta_{\gamma})(e_{\pm \beta}) \ne xe_{\pm \beta}$. This implies that  $\Psi(\zeta_{\beta}) \neq  \Psi(\zeta_{\gamma})$, and hence by Theorem \ref{linear rep VAG}, we have $\zeta_\beta\ne \zeta_\gamma \mod PVA[\Gamma]'$. Now, suppose that $\gamma=-\beta$. Let $w\in W[\Gamma]$ and $s \in S$ be such that $\rho(w)(\alpha_s)=\beta$. Then, we see that $\zeta_{\beta}= \iota_W(w)(\tau_s\sigma_s)\iota_W(w)^{-1}$ and $\zeta_{-\beta}= \iota_W(w)(\sigma_s\tau_s)\iota_W(w)^{-1}$. Since  $\tau_s\sigma_s\tau_s\sigma_s^{-1}\neq 1$ in $VA[\Gamma]$ \cite[Proof of Corollary 3.4]{MR4753773}, we get
$$
\zeta_{\beta}\zeta_{-\beta}^{-1}= \iota_W(w)(\tau_s\sigma_s\tau_s \sigma_s^{-1})\iota_W(w)^{-1} \neq 1,
$$
which completes the proof.
\end{proof}

\begin{definition}
Let $\Upsilon$ be a simple graph with vertex set $V(\Upsilon)$ and edge set $E(\Upsilon)$. Let $G_v$ be a group assigned to a vertex $v \in V(\Upsilon)$. Then the graph product $G(\Upsilon)$ is defined to be the quotient of the free product of $*_{v \in V(\Upsilon)} \,G_v$ by the  set of relations
$$
\big\{ [G_v, G_w]=1 \mid  \text{whenever}~ (v,w) \in E(\Upsilon) \big\}.
$$
\end{definition}
\begin{remark}
It is easy to see that each right-angled virtual Artin group is a graph product of groups with each vertex group being either  $\mathbb Z$ or $\mathbb Z_2$. Hsu and Wise proved in \cite{MR1704150} that a graph product of subgroups of  Coxeter groups can be embedded as a subgroup of another Coxeter group, and hence such a graph product is linear.
\end{remark}

Since right-angled virtual Artin groups are linear, the following seems a natural question.

\begin{question}
Can we classify the virtual Artin groups that admit faithful linear representations?
\end{question}
\medskip

\section{Crystallographic quotient of $VA[\Gamma]$}\label{sec crystallographic quotient}

 A closed subgroup $H$ of a Hausdorff topological group $G$ is called {\it uniform} if $G/H$ is a compact space.  A discrete and uniform subgroup $G$ of $~\mathbb{R}^n \rtimes \Oo(n, \mathbb{R})$ is called a {\it crystallographic group} of dimension $n$. In addition, if $G$ is torsion-free, then it is called a {\it Bieberbach group} of dimension $n$. The following characterisation of crystallographic groups is well-known \cite[Lemma 8]{MR3595797}.

\begin{lemma}
 A group $G$ is a crystallographic group if and only if there is an integer $n$, a finite group $H$ and a short exact sequence
$$
0 \longrightarrow \mathbb{Z}^n \longrightarrow G \stackrel{\eta}{\longrightarrow} H \longrightarrow 1
$$
such that the integral representation $\Theta : H \longrightarrow \GL(n, \mathbb{Z})$, defined by $\Theta(h)(x)=z x z^{-1}$, is faithful. Here,  $h \in H$, $x \in \mathbb{Z}^n$ and $z \in G$ is such that $\eta(z)=h$. 
\end{lemma}

The group $H$ is referred to as the {\it holonomy group} of $G$, the integer $n$ is known as the {\it dimension} of $G$, and $\Theta$ is termed the {\it holonomy representation} of $G$. It is well-known that any finite group is the holonomy group of some flat manifold \cite[Theorem III.5.2]{MR0862114}. Also, there is a correspondence between the class of Bieberbach groups and the class of compact flat Riemannian manifolds \cite[Theorem 2.1.1]{MR1482520}. Virtual braid groups, virtual twin groups and virtual triplet groups are three canonical extensions of symmetric groups that are motivated by virtual knot theories. Crystallographic quotients of these families of groups have been considered in \cite[Theorem 2.4 and Theorem 3.5]{MR4607569} and \cite[Theorem 5.4]{MR4750162}. Regarding virtual Artin groups, we derive the following result.

\begin{theorem}\label{thm crystallographic SVAG}
Let $W[\Gamma]$ be a spherical Coxeter group, and let $\Phi[\Gamma]$ denote its root system, with $|\Phi[\Gamma]|=n$. Then there is a split exact sequence
$$
1 \to \mathbb{Z}^{n} \to VA[\Gamma]/PVA[\Gamma]' \to  W[\Gamma] \to 1
$$
such that the group $VA[\Gamma]/PVA[\Gamma]'$ is a crystallographic group of dimension $n$ with holonomy group $W[\Gamma]$.
\end{theorem}

\begin{proof}
By Proposition \ref{prop:freeabelian} and sequence \eqref{splitting for crystallographic}, we obtain the split exact sequence
$$
1 \to \mathbb{Z}^n \to VA[\Gamma]/PVA[\Gamma]' \to  W[\Gamma] \to 1,
$$
where $\mathbb{Z}^n \cong PVA[\Gamma]/PVA[\Gamma]'$.  Let $\Theta: W[\Gamma] \to \GL(n, \mathbb{Z})$ be the natural action of $W[\Gamma]$ on $\mathbb{Z}^n$, which is induced by the action given in Remark \ref{rmk:action}. Suppose that $\Theta(w)=\id$ for some $w \in W[\Gamma]$. But, $\zeta_{\rho(w)(\beta)} = \zeta_\beta \mod PVA[\Gamma]'$ for all $\beta \in \Phi[\Gamma]$ if and only if  $\rho(w)(\beta)=\beta$ for all $\beta \in \Phi[\Gamma]$. Since $\rho$ is faithful, it follows that $w=1$. Hence, the holonomy representation $\Theta$ is faithful, and $VA[\Gamma]/PVA[\Gamma]'$ is a crystallographic group of dimension $n$.
\end{proof}

\subsection{Torsion in the crystallographic quotient of $VA[\Gamma]$}
Next, we identify torsion elements in the crystallographic quotient $VA[\Gamma] /PVA[\Gamma]'$ when $VA[\Gamma]$ is spherical. By \eqref{splitting for crystallographic}, we have $VA[\Gamma] /PVA[\Gamma]' \cong \big(PVA[\Gamma]/PVA[\Gamma]'\big) \rtimes W[\Gamma]$. Thus, any $w \in VA[\Gamma]/PVA[\Gamma]'$ can be written uniquely in the form $w= (\prod_{\beta \in \Phi[\Gamma]} \overline{\zeta}_{\beta}^{\,a_\beta})\theta$, where $\prod_{\beta \in \Phi[\Gamma]} \overline{\zeta}_{\beta}^{\, a_\beta} \in PVA[\Gamma] /PVA[\Gamma]'$, $a_\beta \in \mathbb Z$ and $\theta \in W[\Gamma]$. The element $w$ acts on the set $\{\zeta_\beta ~|~ \beta \in \Phi[\Gamma]\}$ via the action of its image $\overline{\pi}_P(w)=\theta$. In view of Remark \ref{rmk:action}, we have
$$w \cdot \zeta_\beta= \theta \cdot \zeta_\beta= \zeta_{\rho(\theta)(\beta)}.$$ We denote the orbit of an element $\zeta_\beta$ under the action of $w$  by $\mathcal{O}_\theta(\zeta_\beta)$ and denote a set of representatives of orbits of the action of  $w$ on $\{\zeta_\beta ~|~ \beta \in \Phi[\Gamma]\}$ by $T_\theta$.  Adapting the approach from \cite[Theorem 2.1]{MR4607569}, we obtain the following result.
\par

\begin{theorem}\label{crystallographic torsion}
Let $VA[\Gamma]$ be a spherical virtual Artin group.  Let $w= \big(\prod_{\beta \in \Phi[\Gamma]} \overline{\zeta}_{\beta}^{\,a_\beta} \big) \theta$ be an element of $VA[\Gamma]/PVA[\Gamma]'$ and $T_\theta$ a set of representatives of orbits of the action of  $w$ on the set $\{\zeta_\beta ~|~ \beta \in \Phi[\Gamma]\}$. Then $w$ has order $t$ if and only if $\theta$ has order $t$ and $\sum_{\zeta_\beta \in \mathcal{O}_\theta(\zeta_\gamma)} a_\beta=0$ for all $\zeta_\gamma \in T_\theta$.
\end{theorem}

\begin{proof}
We have
\begin{eqnarray*}
& & \Bigg( (\prod_{\beta \in \Phi[\Gamma]} \overline{\zeta}_{\beta}^{\,a_\beta} ) \theta\Bigg)^t \\
&=& \Bigg(\prod_{\beta \in \Phi[\Gamma]} \overline{\zeta}_{\beta}^{\, a_\beta} \Bigg) \Bigg(  \theta (\prod_{\beta \in \Phi[\Gamma]} \overline{\zeta}_{\beta}^{\, a_\beta} ) \theta^{-1} \Bigg) \Bigg(\theta^2 (\prod_{\beta \in \Phi[\Gamma]} \overline{\zeta}_{\beta}^{\, a_\beta} )  \theta^{-2} \Bigg) \cdots \Bigg( \theta^{t-1} (\prod_{\beta \in \Phi[\Gamma]} \overline{\zeta_\beta}^{a_\beta} ) \theta^{-(t-1)} \Bigg) \theta^t \\
&=& \Bigg(\prod_{\beta \in \Phi[\Gamma]} \overline{\zeta}_{\beta}^{\,a_\beta} \Bigg) \Bigg(\prod_{\beta \in \Phi[\Gamma]} \overline{\zeta}_{{\rho(\theta)(\beta)}}^{\,a_\beta} \Bigg) \Bigg(\prod_{\beta \in \Phi[\Gamma]} \overline{\zeta}_{{\rho(\theta)^2(\beta)}}^{\,a_\beta} \Bigg) \cdots \Bigg( \prod_{\beta \in \Phi[\Gamma]} \overline{\zeta}_{{\rho(\theta)^{t-1}(\beta)}}^{\,a_\beta} \Bigg) \theta^t.\\
& =& \prod_{\beta \in \Phi[\Gamma]} \bigg(\overline{\zeta}_{\beta}^{\, a_\beta}\, \overline{\zeta}_{\beta}^{\, a_{\rho(\theta)^{-1}(\beta)}}\, \overline{\zeta}_{\beta}^{\,a_{\rho(\theta)^{-2}(\beta)}} \, \cdots \, \overline{\zeta}_{\beta}^{\,a_{\rho(\theta)^{1-t}(\beta)}}\bigg) \theta^t\\
& =& \prod_{\beta \in \Phi[\Gamma]} \bigg(\overline{\zeta}_{\beta}^{\, {a_\beta}+a_{\rho(\theta)^{-1}(\beta)}+a_{\rho(\theta)^{-2}(\beta)}+ \cdots +a_{\rho(\theta)^{1-t}(\beta)}}\bigg) \theta^t.
\end{eqnarray*}
The last expression implies that the total exponent sum of the generator $\overline{\zeta}_{\beta}$ is the same as that of the generator $\overline{\zeta}_{{\rho(\theta)^\ell(\beta)}}$ for each $0 \le \ell \le t-1$.  We see that 
$$\Bigg( (\prod_{\beta \in \Phi[\Gamma]} \overline{\zeta}_{\beta}^{\,a_\beta}) \theta \Bigg)^t=1$$
if and only if $\theta^t=1$ and
$$
\sum_{0 \le \ell \le t-1} a_{\rho(\theta)^\ell(\beta)}=0 ~~\text{for all}~~\beta \in \Phi[\Gamma].$$
Using the orbit representatives from $T_\theta$,  the latter is equivalent to 
$$\sum_{\zeta_\beta \in \mathcal{O}_\theta(\zeta_{\gamma})} a_\beta=0~~\text{for all}~~\zeta_{\gamma} \in T_\theta,
$$
which completes the proof.
\end{proof}

\subsection{Conjugacy in the crystallographic quotient of $VA[\Gamma]$} 
Let $VA[\Gamma]$ be a spherical virtual Artin group. Let $w=(\prod_{\beta \in \Phi[\Gamma]} \overline{\zeta}_{\beta}^{\,a_\beta})\theta$ and  $w'=(\prod_{\beta \in \Phi[\Gamma]} \overline{\zeta}_{\beta}^{\, b_\beta})\theta'$ be elements of $ VA[\Gamma]/PVA[\Gamma]'$. If $\gamma\in VA[\Gamma]/PVA[\Gamma]'$ such that $\gamma w \gamma^{-1}=w'$, then we have $\overline{\pi}_P(\gamma)\, \theta \, \overline{\pi}_P(\gamma)^{-1}=\theta'$. On the other hand, if  $\theta$ and $\theta'$ are conjugate by an element of $W[\Gamma]$, then by choosing a lift  $\tilde{\gamma}$ of this conjugating element in $ VA[\Gamma]/PVA[\Gamma]'$, we can write $\tilde{\gamma} w'\tilde{\gamma}^{-1}=\prod_{\beta \in \Phi[\Gamma]} (\overline{\zeta}_{\beta}^{\,b_\beta'})\theta$ for some $b_\beta' \in \mathbb{Z}$. Thus, without loss of generality, we can assume that $w'=(\prod_{\beta \in \Phi[\Gamma]} \overline{\zeta}_{\beta}^{\,b_\beta})\theta$ 

\begin{theorem}\label{crystallographic conjugacy}
Let $VA[\Gamma]$ be a spherical virtual Artin group. The elements $w=(\prod_{\beta \in \Phi[\Gamma]} \overline{\zeta}_{\beta}^{\,a_\beta})\theta$ and $w'=(\prod_{\beta \in \Phi[\Gamma]} \overline{\zeta}_{\beta}^{ \,b_\beta})\theta$ are conjugate in $VA[\Gamma]/PVA[\Gamma]'$ if and only if 
$$\sum_{0\le t \le m_\beta-1}f_{\rho(\theta)^{t}(\beta)}+c_{\rho(\overline{\pi}_P(\eta))^{-1}(\beta)}=d_\beta$$
where $\zeta_\beta \in T_\theta$, $m_\beta=|\mathcal{O}_{\theta}(\zeta_{\beta})|$, $c_{\gamma}=\sum_{\zeta_\beta\in\mathcal{O}_{\theta}(\zeta_{\gamma})}{a_\beta}$, $d_{\gamma}=\sum_{\zeta_\beta\in\mathcal{O}_{\theta}(\zeta_{\gamma})}{b_\beta}$, $c_{\rho(\overline{\pi}_P(\eta))^{-1}(\gamma)}=\sum_{\zeta_\beta\in\mathcal{O}_{\theta}(\zeta_{\rho(\overline{\pi}_P(\eta))^{-1}(\gamma)})}{a_\beta}$ for $\zeta_\gamma \in T_\theta$ and the element  $\eta\in VA[\Gamma]/PVA[\Gamma]'$ is chosen such that $\overline{\pi}_P(\eta)$ is in the centralizer of $\theta$ in $W[\Gamma]$ with $\eta\theta\eta^{-1}= (\prod_{\beta\in \Phi[\Gamma]} \overline{\zeta}_{\beta}^{\,f_{\beta}})\theta$.
\end{theorem}
\begin{proof}
Let $T_\theta=\{\zeta_{\beta_1}, \ldots, \zeta_{\beta_r}\}$ be a set of representatives of orbits of the action of $\theta$ on $\{\zeta_\beta ~|~ \beta \in \Phi[\Gamma]\}$. We claim that $w$ is conjugate to  $(\prod_{\zeta_{\beta_i}\in T_{\theta}} \overline{\zeta}_{\beta_i}^{\,c_{\beta_i}})\theta$, where $c_{\beta_i}=\sum_{\zeta_\beta\in\mathcal{O}_{\theta}(\zeta_{\beta_i})}{a_\beta}$ for each $\zeta_{\beta_i}\in T_{\theta}$, whereas $w'$ is conjugate to $(\prod_{\zeta_{\beta_i}\in T_{\theta}} \overline{\zeta}_{\beta_i}^{\,d_{\beta_i}})\theta$, where $d_{\beta_i}=\sum_{\zeta_\beta\in\mathcal{O}_{\theta}(\zeta_{\beta_i})}{b_\beta}$ for each $\zeta_{\beta_i}\in T_{\theta}$
\par

Let $Y=\prod_{\beta \in \Phi[\Gamma]}\overline{\zeta}_{\beta}^{\, y_\beta}\in PVA[\Gamma]/PVA[\Gamma]'$, where $y_\beta\in\mathbb Z$. Then we have
\begin{eqnarray*}
& &Y \Bigg(\big(\prod_{\beta \in \Phi[\Gamma]}\overline{\zeta}_{\beta}^{\,a_\beta} \big)\theta \Bigg)Y^{-1}=\big(\prod_{\zeta_{\beta_i}\in T_{\theta}} \overline{\zeta}_{\beta_i}^{\,c_{\beta_i}} \big)\theta\\
\Leftrightarrow & & Y \Bigg( \big(\prod_{\beta \in \Phi[\Gamma]} \overline{\zeta}_{\beta}^{\,a_\beta} \big)\theta  \Bigg)Y^{-1} \theta^{-1}=\prod_{\zeta_{\beta_i}\in T_{\theta}} \overline{\zeta}_{\beta_i}^{\,c_{\beta_i}}\\
\Leftrightarrow & & Y \Bigg(\prod_{\beta \in \Phi[\Gamma]} \overline{\zeta}_{\beta}^{\,a_\beta}\Bigg)  \Bigg(\prod_{\beta \in \Phi[\Gamma]}\overline{\zeta}_{{\rho(\theta)(\beta)}}^{\,-y_\beta} \Bigg)=\prod_{\zeta_{\beta_i}\in T_{\theta}}\overline{\zeta}_{\beta_i}^{\,c_{\beta_i}}
\end{eqnarray*}
\begin{align}\label{se1}
\Leftrightarrow\begin{cases}
	y_{\beta}+a_{\beta}- y_{\rho(\theta)^{-1}(\beta)}=0 &~\mbox{ if }~\zeta_\beta\notin T_{\theta},\\
	y_{\beta_i}+a_{\beta_i}-y_{\rho(\theta)^{-1}(\beta_i)}=c_{\beta_i} &~\mbox{ if }~\zeta_\beta=\zeta_{\beta_i}\in T_{\theta}.
\end{cases}
\end{align}
The system of equations \eqref{se1} has a solution if and only if the following subsystem of equations has a solution   
	\begin{align}\label{se2}
\begin{cases}
		y_{\rho(\theta)^t(\beta_i)}+a_{\rho(\theta)^t(\beta_i)}-y_{\rho(\theta)^{-(1+m_{i}-t)}(\beta_i)}=0,\\
		y_{\beta_i}+a_{\beta_i}-y_{\rho(\theta)^{-1}(\beta_i)}=c_{\beta_i},
	\end{cases}
\end{align}
for all $1\le t<m_{i}$ and $\zeta_{\beta_i}\in T_{\theta}$, where $m_{i}=|\mathcal{O}_{\theta}(\zeta_{\beta_i})|$.  We show that the subsystems of equations \eqref{se2} admits a solution. For fixed $\zeta_{\beta_i}\in T_{\theta}$, consider the elements $e_1=\zeta_{\beta_i},e_2=\zeta_{\rho(\theta)(\beta_i)},\dots,e_{m_{i}}=\zeta_{\rho(\theta)^{m_{i}-1}(\beta_i)}$. Then the  matrix of $\theta$ with respect to the ordered set $\{e_1,e_2,\dots,e_{m_i}\}$ is given by
$$M_{\beta_i}=
\left(\begin{matrix}	0&0&\ldots&0&1\\
	1&0&\ldots&0&0\\
	0&1&\ldots&0&0\\
	\vdots&\vdots&\ddots&\vdots&\vdots\\ 
	0&0&\ldots&1&0
\end{matrix}\right).
$$
Thus, for each $\zeta_{\beta_i}\in T_{\theta}$, the subsystems of equations \eqref{se2} can be written as
$$
[I_{\beta_i}-M_{\beta_i}]
\left(\begin{matrix}
	y_{\beta_i}\\
	y_{\rho(\theta)(\beta_i)}\\
	\vdots\\
	y_{\rho(\theta)^{m_{i}-1}(\beta_i)}
\end{matrix}\right)=
\left(\begin{matrix}	-a_{\beta_i}+c_{\beta_i}\\
	-a_{\rho(\theta)(\beta_i)}\\
	\vdots\\
	-a_{\rho(\theta)^{m_{i}-1}(\beta_i)}
\end{matrix}\right),
$$ 
where $I_{\beta_i}$ is the $m_{i}\times m_{i}$ identity matrix and $m_i=|\mathcal{O}_{\theta}(\zeta_{\beta_i})|$. Analysing the matrix, we conclude that \eqref{se2}) has a solution if and only if $$c_{\beta_i}=\sum_{\zeta_\beta\in\mathcal{O}_{\theta}(\zeta_{\beta_i})}{a_\beta}=\sum_{0 \le \ell \le m_{i}-1} a_{\rho(\theta)^\ell(\beta_i)}$$ for all $\zeta_{\beta_i}\in T_{\theta}$. Similar arguments establish the second assertion of the claim. 
\par

Next, we analyse the conditions under which  $(\prod_{\zeta_{\beta_i}\in T_{\theta}} \overline{\zeta}_{\beta_i}^{\, c_{\beta_i}})\theta$ and $(\prod_{\zeta_{\beta_i}\in T_{\theta}} \overline{\zeta}_{\beta_i}^{\, d_{\beta_i}})\theta$ are conjugate in $VA[\Gamma]/PVA[\Gamma]'$. We choose $\eta\in VA[\Gamma]/PVA[\Gamma]'$ such that $\overline{\pi}_P(\eta)$ is in the centralizer of $\theta$ in $W[\Gamma]$. This gives $\eta\theta\eta^{-1}= (\prod_{\beta\in \Phi[\Gamma]} \overline{\zeta}_{\beta}^{\, f_{\beta}})\theta$ for some $f_{\beta_i} \in \mathbb{Z}$. Let $Y=\prod_{\beta \in \Phi[\Gamma]}\overline{\zeta}_{\beta}^{\,y_\beta}\in PVA[\Gamma]/PVA[\Gamma]'$, where $y_\beta\in\mathbb Z$. Then we have
\begin{small}
\begin{eqnarray}
\nonumber & & Y\eta \Bigg(\prod_{\zeta_{\beta_i}\in T_{\theta}} \overline{\zeta}_{\beta_i}^{\, c_{\beta_i}} \Bigg) \theta\eta^{-1}Y^{-1}=(\prod_{\zeta_{\beta_i}\in T_{\theta}} \overline{\zeta}_{\beta_i}^{\,d_{\beta_i}})\theta\\
\nonumber  \Leftrightarrow & & Y\eta \Bigg(\prod_{\zeta_{\beta_i}\in T_{\theta}}\overline{\zeta}_{\beta_i}^{\,c_{\beta_i}}\Bigg) \theta\eta^{-1}Y^{-1}\theta^{-1}=\prod_{\zeta_{\beta_i}\in T_{\theta}} \overline{\zeta}_{\beta_i}^{\, d_{\beta_i}}\\
\nonumber  \Leftrightarrow & & Y \Bigg(\prod_{\zeta_{\beta_i}\in T_{\theta}} \overline{\zeta}_{{\rho(\overline{\pi}_P(\eta))(\beta_i)}}^{\,c_{\beta_i}}\Bigg) \eta\theta\eta^{-1}Y^{-1}\theta^{-1}=\prod_{\zeta_{\beta_i}\in T_{\theta}} \overline{\zeta}_{\beta_i}^{\,d_{\beta_i}} ,\\
\nonumber  \Leftrightarrow & & Y \Bigg(\prod_{\zeta_{\beta_i}\in T_{\theta}} \overline{\zeta}_{{\rho(\overline{\pi}_P(\eta))(\beta_i)}}^{\, c_{\beta_i}}\Bigg) \Bigg(\prod_{\beta \in \Phi[\Gamma]} \overline{\zeta}_{\beta}^{\,f_\beta} \Bigg) \theta Y^{-1}\theta^{-1}=\prod_{\zeta_{\beta_i}\in T_{\theta}} \overline{\zeta}_{\beta_i}^{\,d_{\beta_i}},\\
\label{mid equation} \quad \Leftrightarrow & & \Bigg(\prod_{\beta \in \Phi[\Gamma]}\overline{\zeta}_{\beta}^{\,y_\beta} \Bigg) \Bigg(\prod_{\zeta_{\beta_i} \in T_{\theta}} \overline{\zeta}_{{\rho(\overline{\pi}_P(\eta))(\beta_i)}}^{\,c_{\beta_i}} \Bigg) \Bigg(\prod_{\beta \in \Phi[\Gamma]} \overline{\zeta}_{\beta}^{\,f_\beta} \Bigg) \Bigg(\prod_{\beta \in \Phi[\Gamma]}\overline{\zeta}_{{\rho(\theta)(\beta)}}^{\, -y_\beta} \Bigg)=\prod_{\zeta_{\beta_i}\in T_{\theta}} \overline{\zeta}_{\beta_i}^{\,d_{\beta_i}}.
\end{eqnarray}
\end{small}

We have defined $c_{\beta_i}$ for each $\zeta_{\beta_i} \in T_{\theta}$. Given any $\beta \in \Phi[\Gamma]$, we have $\zeta_\beta\in  \mathcal{O}_{\theta}(\zeta_{\beta_i})$ for some $\zeta_{\beta_i} \in T_\theta$. In this case, we set $c_\beta=c_{\beta_{i}}$.  For each $\zeta_{\beta_i}\in T_\theta$, note that $\rho(\overline{\pi}_P(\eta))^{-1}\cdot\zeta_{\beta_i}=\zeta_{\rho(\overline{\pi}_P(\eta))^{-1}(\beta_i)} \in \mathcal{O}_{\theta}(\zeta_{\beta_j})$ for a unique $\zeta_{\beta_j} \in T_\theta$. This implies that there exists an integer $0\leq l'\leq m_j$ such that  $\rho(\theta)^{l'}\cdot \zeta_{\beta_j}=\rho(\overline{\pi}_P(\eta))^{-1}\cdot\zeta_{\beta_i}$. Since  $\overline{\pi}_P(\eta)$ is in the centralizer of $\theta$, we conclude that $\zeta_{\rho(\overline{\pi}_P(\eta))(\beta_j)}=\zeta_{\rho(\theta)^{-l'}(\beta_i)}$. Setting $l = -l' \mod m_i$, we obtain $\zeta_{\rho(\overline{\pi}_P(\eta))(\beta_j)}=\zeta_{\rho(\theta)^{l}(\beta_i)}$. In view of Proposition \ref{prop:freeabelian}, this further gives  $\rho(\overline{\pi}_P(\eta))(\beta_j)=\rho(\theta)^l(\beta_i)$.
\par

It is easy to see that $\zeta_{\beta_j}\in \mathcal{O}_{\theta}(\zeta_{\rho(\overline{\pi}_P(\eta))^{-1}(\beta_i)})$, and hence $c_{\beta_j} = c_{\rho(\overline{\pi}_P(\eta))^{-1}(\beta_i)}$. If $l=0$, then equation \eqref{mid equation} holds if and only if the following system of equations has a solution 
\begin{align}\label{mid equation 1}
\begin{cases}
y_{\rho(\theta)^k(\beta_i)}+f_{\rho(\theta)^k(\beta_i)}-y_{\rho(\theta)^{-(1+m_i-k)}(\beta_i)}=0,\\
y_{\beta_i}+f_{\beta_i}+c_{\rho(\overline{\pi}_P(\eta))^{-1}(\beta_i)}-y_{\rho(\theta)^{-1}(\beta_i)}=d_{\beta_i},	
\end{cases} 
\end{align}
for $1\le k<m_i$, where $m_i=|\mathcal{O}_{\theta}(\zeta_{\beta_i})|$. If $1\leq l \leq m_i$, then equation \eqref{mid equation} holds if and only if the following system of equations has a solution 
\begin{align}\label{mid equation 2}
\begin{cases}
y_{\rho(\theta)^k(\beta_i)}+f_{\rho(\theta)^k(\beta_i)}-y_{\rho(\theta)^{-(1+m_i-k)}(\beta_i)}=0,\\
y_{\rho(\theta)^l(\beta_i)}+f_{\rho(\theta)^l(\beta_i)}+c_{\rho(\overline{\pi}_P(\eta))^{-1}(\beta_i)}-y_{\rho(\theta)^{-(1+m_i-l)}(\beta_i)}=0,\\
y_{\beta_i}+f_{\beta_i}-y_{\rho(\theta)^{-1}(\beta_i)}=d_{\beta_i},	
\end{cases} 
\end{align}
for $1\le k<m_i$ and $k\not=l$, where $m_i=|\mathcal{O}_{\theta}(\zeta_{\beta_i})|$.  For fixed $\zeta_{\beta_i}\in T_{\theta}$, if $l=0$, then the system of equations \eqref{mid equation 1} can be written as
$$
[I_{\beta_i}-M_{\beta_i}]
\left(\begin{matrix}
	y_{\beta_i}\\
	y_{\rho(\theta)(\beta_i)}\\
	\vdots\\
	y_{\rho(\theta)^{m_{i}-1}(\beta_i)}
\end{matrix}\right)=
\left(\begin{matrix}	-f_{\beta_i}-c_{\rho(\overline{\pi}_P(\eta))^{-1}(\beta_i)}+d_{\beta_i}\\
	-f_{\rho(\theta)(\beta_i)}\\
	\vdots\\
	-f_{\rho(\theta)^{m_{i}-1}(\beta_i)}
\end{matrix}\right).
$$ 
Similarly, if $1\leq l \leq m_i$, then the system of equations \eqref{mid equation 2} can be written as
$$
[I_{\beta_i}-M_{\beta_i}]
\left(\begin{matrix}
	y_{\beta_i}\\
	\vdots\\
		y_{\rho(\theta)^l(\beta_i)}\\
	\vdots\\
	y_{\rho(\theta)^{m_{i}-1}(\beta_i)}
\end{matrix}\right)=
\left(\begin{matrix}	-f_{\beta_i}+d_{\beta_i}\\
	\vdots\\
		-f_{\rho(\theta)^l(\beta_i)}-c_{\rho(\overline{\pi}_P(\eta))^{-1}(\beta_i)}\\
	\vdots\\
	-f_{\rho(\theta)^{m_{i}-1}(\beta_i)}
\end{matrix}\right).
$$ 
Analysing the matrices, we conclude that \eqref{mid equation 1} and \eqref{mid equation 2} have a solution if and only if $\sum_{0\le t \le m_i-1}f_{\rho(\theta)^{t}(\beta_i)}+c_{\rho(\overline{\pi}_P(\eta))^{-1}(\beta_i)}=d_{\beta_i}$ for all $\zeta_{\beta_i}\in T_{\theta}$. This completes the proof of the theorem.
\end{proof}

\begin{remark}
Theorems \ref{crystallographic torsion} and \ref{crystallographic conjugacy} remain valid even if $VA[\Gamma]$ is not spherical or equivalently $\Phi[\Gamma]$ is infinite. In this case, the set $T_\theta$ of representatives of the orbits as well as the number of conditions in  Theorem \ref{crystallographic torsion} becomes infinite. Since $a_\beta = 0$ for all but finitely many $\beta$, only finitely many of these conditions are non-trivial. 
\par
Similarly,  the matrices $M_\beta$ may become infinite size in Theorem \ref{crystallographic conjugacy},  the corresponding system of equations still involves only finitely many non-zero variables and non-zero matrix entries. Hence, the arguments in the proofs remain valid.
\end{remark}
\medskip

\section{$R_\infty$-property of right-angled virtual Artin groups} \label{section twisted conjugacy}
In this section, we investigate twisted conjugacy in right-angled virtual Artin groups. There are many ways to associate graphs to Coxeter groups. In line with the standard convention, we use the following graph for right-angled Coxeter groups (which is not the Coxeter graph considered in the preceding sections).
\par

Let $\Upsilon$ be a simple graph with the vertex set $V(\Upsilon)$ and the edge set $E(\Upsilon)$. We can define a right-angled Coxeter group $W(\Upsilon)$ corresponding to $\Upsilon$ by the presentation 
$$
W(\Upsilon)= \Bigg\langle V(\Upsilon)~ \mid~ s^2=1~ \textrm{for all}~s \in V(\Upsilon) ~\textrm{and}~ s t= t s~\textrm{whenever}~ (s, t)  \in E(\Upsilon) \Bigg\rangle.
$$
Then, the corresponding right-angled virtual Artin group  $VA(\Upsilon)$ has the presentation 
\begin{eqnarray*}
VA(\Upsilon) &=& \Bigg\langle \sigma_s, \tau_s;~ s \in V(\Upsilon) \mid~ \tau_s^2=1~ \textrm{for all}~s \in V(\Upsilon)~\textrm{and}~\\
&&  \tau_s \tau_t= \tau_t \tau_s, ~\sigma_s \sigma_t= \sigma_t \sigma_s, ~ \tau_s \sigma_t= \sigma_t \tau_s~\textrm{whenever}~ (s, t)  \in E(\Upsilon) \Bigg\rangle.
\end{eqnarray*}

Given $\Upsilon$, consider the graph $\widetilde{\Upsilon}$ with the vertex set $V(\Upsilon) \times \{0,1\}$ and with the edge set described as follows:
\begin{enumerate}
\item $(s,0)$ and $(t,0)$ are connected by an edge in $\widetilde{\Upsilon}$ if and only if $(s, t) \in E(\Upsilon)$.
\item $(s,1)$ and $(t,1)$ are connected by an edge  in $\widetilde{\Upsilon}$ if and only if $(s, t) \in E(\Upsilon)$.
\item $(s,0)$ and $(t,1)$ are connected by an edge  in $\widetilde{\Upsilon}$ if and only if $(s, t) \in E(\Upsilon)$.
\end{enumerate}
If we label each vertex $(s,0)$ with $\mathbb Z$ and each vertex $(s,1)$ with $\mathbb Z_2$, then the graph product  $G(\widetilde{\Upsilon})$ is isomorphic to the right-angled virtual Artin group $VA(\Upsilon)$.
\medskip

\subsection{A generating set for the automorphism group of a graph product of cyclic groups}
Let $\Upsilon$ be a simple graph. The link $lk(v)$ of a vertex $v \in V(\Upsilon)$ is defined as the set of all the vertices that are connected to $v$ by an edge from $E(\Upsilon)$. The star  $st(v)$ of $v$ is defined to be $lk(v) \cup \{v\}$. It follows from \cite[Main Theorem]{MR2900856} that the automorphism group of a graph product $G(\Upsilon)$ of cyclic groups  is generated by the following four types of automorphisms: 

\begin{enumerate}
\item Labeled graph automorphism: It is an automorphism of $G(\Upsilon)$ induced by an automorphism $\gamma:\Upsilon \to \Upsilon$ of the graph $\Upsilon$, such that $G_v \cong G_{\gamma(v)}$ for every vertex $v \in V(\Upsilon)$.
\item Inversion $\iota_a$: It sends a generator $a$ to its inverse and leaves all other generators fixed.
\item Transvection $\tau_{a,b}$: It sends a generator $a$ to $ab$ and leaves all other generators fixed, where $b$ is another generator satisfying the following conditions. 
\item[] If $a$ is an infinite order generator, then $lk(a) \subset lk(b)$. 
\item[] If $a$ is a finite order generator, then $b$ is a finite order generator such that $st(a) \subset st(b)$ and the order of $b$ divides the order of $a$.
\item Partial conjugation $p_{b, C}$: If $b$ is a generator and $C$ is a connected component of $\Upsilon \setminus  st(b)$, then  $p_{b, C}$ sends each generator $a$ in $C$ to $ba b^{-1}$ and leaves the other generators fixed. We see that if $\Upsilon \setminus st(b)$ is connected, that is, $C=\Upsilon \setminus  st(b)$, then the partial conjugation $p_{b, C}$ is simply the inner automorphism induced by $b$.
\end{enumerate}

\begin{remark}\label{rmk:trans}
The transvections that are required in the generating set of the automorphism group of a right-angled virtual Artin group are given as follows:
\begin{itemize}
\item $\tau_{\sigma_t, \ast}$, where $\ast \in \{\sigma_s, \tau_s \mid s \in V(\Upsilon)\}\setminus \{\sigma_t\}$ such that $lk(\sigma_t)$ is a subset of the star of the vertex corresponding to the generator $\ast$.
\item $\tau_{\tau_s, \tau_t}$, where $st(\tau_s) \subset st(\tau_t)$.
\end{itemize}
\end{remark}
\medskip

\subsection{The $R_\infty$-property}
Let $G$ be a group and $\phi \in \Aut(G)$. Two elements $x, y\in G$ are said to be {\it $\phi$-conjugate} if there exists an element $g\in G$ such that $x = gy\phi(g^{-1})$. The equivalence classes under the relation of $\phi$-conjugation are called $\phi$-conjugacy classes. Taking $\phi$ to be the identity automorphism gives the usual conjugacy classes. The number  $R(\phi)$ of $\phi$-conjugacy classes  is called the \textit{Reidemeister number} of the automorphism $\phi$, where $R(\phi)\in \mathbb{N} \cup \{\infty\}$. We say that the group $G$ has the {\it $R_{\infty }$-property} if $R(\phi)= \infty$ for each $\phi \in \Aut(G)$. The idea of twisted conjugacy arose from the work of Reidemeister \cite{Reidemeister}, and  the study of the $R_{\infty }$-property in groups has garnered significant attention, largely due to its strong connection with Nielsen fixed-point theory. It has been proved recently in \cite[Theorem 3]{DekimpeDacibergOcampo2024} that the virtual braid group $VB_n$ has the $R_\infty$-property for each $n \ge 2$. For right-angled virtual Artin groups, we prove the following result.

\begin{theorem}\label{R infinity theorem}
Each right-angled virtual Artin group has the $R_\infty$-property.
\end{theorem}

\begin{proof}
Let $\Upsilon$ be a simple graph, and $VA(\Upsilon)$ the associated right-angled virtual Artin group. First suppose that $\Upsilon$ is a non-complete graph. Setting $K:=\llangle \tau_s \mid s\in V(\Upsilon)  \rrangle$, we see that $K$ is invariant under all labeled graph automorphisms, inversions and partial conjugations. Further, in view of Remark \ref{rmk:trans}, we see that $K$ is also invariant under transvections. Consequently, $K$ is a characteristic subgroup of $VA(\Upsilon)$. In addition, the quotient $VA(\Upsilon)/K$ is the right-angled Artin group associated to the graph $\Upsilon$ \cite[Theorem 3.1]{MR2012965}. By \cite[Theorem 1.6]{Wit}, every right-angled Artin group associated to a non-complete graph has the $R_\infty$-property. An elementary observation \cite[Lemma 2.1]{MR3150726} shows that a group has the $R_\infty$-property if its quotient by a characteristic subgroup has the  $R_\infty$-property.  Hence, $VA(\Upsilon)$ has the $R_\infty$-property.
\par

Next, suppose that $\Upsilon$ is a complete graph on $n \ge 1$ vertices. Then the right-angled virtual Artin group  $VA(\Upsilon)$ is the $n$-fold direct product $(\mathbb Z \ast \mathbb Z_2)^n$. It follows from \cite[Theorem 1.1]{MR4808711} that $\Aut((\mathbb Z \ast \mathbb Z_2)^n) \cong \Aut(\mathbb Z \ast \mathbb Z_2)^n \rtimes S_n$. Further, it follows from \cite[Theorem 1]{MR4124669}  that $\mathbb Z \ast \mathbb Z_2$ has the $R_\infty$-property. Finally, \cite[Proposition 5.1.2]{MR4229464} implies that $VA(\Upsilon)$ has the $R_\infty$-property. 
\end{proof}

We conclude with the following question.

\begin{question}
Can we classify virtual Artin groups that possess the $R_\infty$-property?
\end{question}
\medskip

\noindent\textbf{Acknowledgments.}
The authors would like to thank Anthony Genevois for bringing to their attention earlier works on linearity of right-angled virtual Artin groups. These works had been overlooked by the authors, leading to a reproof of the same results in an earlier (arxiv) version of this paper. NKD acknowledges support from the NBHM grant 0204/1/2023/R \&D-II/1792.  PK is supported by the PMRF fellowship at IISER Mohali.  TKN is supported by the Start-up Research Grant SRG/2023/001556/PMS. MS is supported by the Swarna Jayanti Fellowship grants DST/SJF/MSA-02/2018-19 and SB/SJF/2019-20/04. 
\medskip

\noindent\textbf{Declaration.}
The authors declare that there is no data associated to this paper and that there are no conflicts of interests.

\bigskip
\end{document}